\documentclass[11pt]{article}
\usepackage[utf8]{inputenc}
\usepackage[numbers]{natbib}

\usepackage{amsmath,amsthm,amsopn,amstext,amscd,amsfonts,amssymb,mathrsfs}

\usepackage[margin=1.1in]{geometry}
\usepackage{authblk}

\newtheorem{thm}{Theorem}[section]

\newtheorem{prop}[thm]{Proposition}

\numberwithin{figure}{section}

\newcommand{\refT}[1]{Theorem~\ref{#1}}

\newcommand{\refP}[1]{Proposition~\ref{#1}}

\newcommand{\refand}[2]{\ref{#1} and~\ref{#2}}

\newcommand\cC{\mathcal C}

\newcommand\cI{\mathcal I}

\newcommand\cS{{\mathcal S}}

\newcommand{\I}[1]{{\mathbf 1}_{[#1]}}
\newcommand{\p}[1]{{\mathbf P}\left(#1\right)}
\newcommand{\E}[1]{{\mathbf E}\left[#1\right]}

\newcommand{\N}{\mathbb{N}}
\newcommand{\R}{\mathbb{R}}
\newcommand{\Z}{\mathbb{Z}}

\newcommand{\Ber}[1]{\mathrm{Bernoulli}(#1)}

\newcommand{\Poi}[1]{\mathrm{Poi}(#1)}

\newcommand{\eps}{\varepsilon}
\newcommand{\e}{\varepsilon}
\newcommand{\g}{\gamma}
\newcommand{\Ln}[1]{\ln\left(#1\right)}

\newcommand{\eqdist}{\stackrel{\mathcal{L}}{=}}
\newcommand{\dTV}{\mathrm{d}_{\mathrm{TV}}}

\usepackage{natbib}

\usepackage[textsize=tiny]{todonotes}

\newcommand{\RRT}{\textrm{RRT}~}
\newcommand{\RRTs}{\textrm{RRT}s~}

\title{Targeted cutting of random recursive trees}

\author[1]{Laura Eslava}
\author[2]{Sergio I. L\'opez}
\author[2]{Marco L. Ortiz}
\affil[1]{Instituto de Investigaciones en Matem\'aticas Aplicadas y en Sistemas, Universidad Nacional Aut\'onoma de México.}
\affil[2]{Facultad de Ciencias, Universidad Nacional Aut\'noma de M\'exico.}

\date{}

\begin{document}

\maketitle

\begin{abstract} 
We propose a method for cutting down a random recursive tree that focuses on its higher degree vertices. Enumerate the vertices of a random recursive tree of size $n$ according to a decreasing order of their degrees; namely, let $(v^{(i)})_{i=1}^{n}$ be so that $\deg(v^{(1)}) \geq \cdots \geq \deg (v^{(n)})$. The targeted, vertex-cutting process is performed by sequentially removing vertices $v^{(1)}$, $v^{(2)}, \ldots, v^{(n)}$ and keeping only the subtree containing the root after each removal. The algorithm ends when the root is picked to be removed. The total number of steps for this procedure, $X_n^{targ}$, is upper bounded by $Z_{\geq D}$, which denotes the number of vertices that have degree at least as large as the degree of the root. We obtain that the first order growth of $X_n^{targ}$ is upper bounded by $n^{1-\ln 2}$, which is substantially smaller than the required number of removals if, instead, the vertices where selected uniformly at random. More precisely, we prove that $\ln(Z_{\geq D})$ grows as $\ln (n)$ asymptotically and obtain its limiting behavior in probability. Moreover, we obtain that the $k$-th moment of $\ln(Z_{\geq D})$ is proportional to $(\ln(n))^k$.
\end{abstract}

\section{Introduction}

\textit{Random recursive trees} (abbreviated as \RRTs) are rooted trees, where each vertex has a unique label, obtained by the following procedure: Let $T_1$ be a single vertex labeled $1$. For $n>1$ the tree $T_n$ is obtained from the tree $T_{n-1}$ by adding a edge directed from a new vertex labeled $n$ to a vertex with label in $\{1,...,n-1\}$, chosen uniformly at random and independent for each $n$. We say that $T_n$ has size $n$ and that the degree of a vertex $v$ is the number of edges directed towards $v$. 

The idea of cutting random recursive trees was introduced by Meir and Moon \cite{Meir1974}. They studied the following procedure: Start with a random recursive tree on $n$ vertices. Choose an edge at random and remove it, along with the cut subtree that does not contain the root. Repeat until the remaining tree consists only of the root; at which point, we say that the tree has been \textit{deleted}. Let $X_n$ be the number of edge removals needed to delete a \RRT with $n$ vertices. 

By a recursive approach using that the remaining tree after one deletion has itself the distribution of a \RRT of smaller, random size, Meir and Moon proved that the expectation of $X_n$ grows asymptotically
as $\frac{n}{\ln(n)}$. 
Panholzer \cite{Panholzer2004} proposed an extension of this procedure: to study the total cost of the algorithm until deletion; which is the sum of the costs at every step where cutting an edge of a tree of size $n$ has a cost of $n^c$ for some non negative constant $c$ (the original cutting corresponds to $c=0$). By the use of generating functions and recursion, they obtained the asymptotic behavior for the $k$-moment of the total cost and proved, as a corollary, that $\frac{\ln(n)}{n}X_n$ converges to one, in probability. Iksanov and M\"{o}hle \cite{Iksanov2007} obtained the expression of $X_n$ up to its random order; namely, that 
\begin{align}\label{eq:Yn}
Y_n:=\frac{(\ln(n))^2}{n}X_n-\ln(n)-\ln(\ln(n))    
\end{align}
converges weakly to a random variable $Y$ with characteristic function given by
$$\varphi_Y(\lambda):=\exp  \Big\{ i\lambda\ln\vert\lambda\vert-  \frac{\pi\vert\lambda\vert}{2} \Big\}.$$
Their proof is based on the construction of a coupling of $X_n$ with the first passage time of certain random walk; while Drmota et al. \cite{Drmota2009} give a proof of this theorem using recursive methods.

In this work we consider random recursive trees and propose a cutting procedure that corresponds to a targeted cutting (focused on high-degree vertices). We first present our model together with the main result, we then overview the proof strategy and discuss some possible interpretations of this procedure and several related models of tree-deletion.

To define the targeted cutting, let $T_n$ be a \RRT of size $n$ and enumerate its vertices as $(v^{(i)})_{i=1}^{n}$ according to a decreasing order of their degrees; that is, $\deg(v^{(1)})\geq \cdots \geq \deg (v^{(n)})$, breaking ties uniformly at random. The targeted vertex-cutting process is performed by sequentially removing vertices $v^{(1)}$, $v^{(2)}, \ldots, v^{(n)}$ and keeping only the subtree containing the root after each removal (skip a step if the chosen vertex had been previously removed). The procedure ends when the root is picked to be removed. Let $X^{targ}_n$ denote the number of vertex deletions before we select the root to be removed.

Our main theorem is an stochastic upper bound for the deletion time. In doing so, we analyse the number of vertices with degree as large as that of the root; see Section~\ref{sec:degree}.

\begin{thm}\label{thm:attack}
The random number of cuts $X_n^{targ}$ in the targeted cutting of $T_n$ satisfies, for any $\eps>0$, $X_n^{targ}=O_p(n^{\gamma+\eps})$ where $\gamma:=1-\ln 2$. Namely, for each $\delta>0$, there is $M=M(\eps,\delta)>0$ and $N=N(\eps, \delta) \in \N$ such that for all $n\ge N$,
\begin{align}\label{eq:Kn_Op}
    \p{X_n^{targ}> M n^{\gamma+\eps}}<\delta.
\end{align}
\end{thm}

Theorem \ref{thm:attack} gives a precise answer to an expected outcome: the targeted cutting deletion time is significantly smaller compared to the deletion time for uniformly random removals.

In the next section we provide precise statements on a random upper bound for $X_n^{targ}$ from which Theorem~\ref{thm:attack} follows. 

\subsection{A random tail of the degree sequence}\label{sec:degree}

For $d\in \N$, we let $Z_d$ and $Z_{\ge d}$ be the number of vertices in $T_n$ that have degree $d$ and degree at least $d$, respectively (we omit the dependence in $n$ for these variables throughout). The asymptotic joint distribution of $(Z_d)_{d\ge 0}$ is described by the limiting distribution, as $n\to \infty$, of certain urn models in  \cite{Janson05}, while the limiting joint distribution of $(Z_{\ge d + \lfloor \log_2 n\rfloor})_{d\in \Z}$ is described (along suitable subsequences) by an explicit Poisson point process in $\Z\cup \{\infty\}$ in  \cite{Addario-Eslava2017}; the lattice shift by $\log_2 n$ steams from the fact that $\Delta_n/\log_2 n$, the renormalized maximum degree in $T_n$, converges a.s. to 1 \cite{DevroyeLu95}. In contrast, the degree of the root in $T_n$ is asymptotically normal with mean $\ln n$ (note that $\log_2 n\approx 1.4 \ln n$).

In this paper we are interested in the random variable $Z_{\geq D}$, which corresponds to the number of vertices that have degree at least the degree of the root of $T_n$, henceforth denoted by $D=D(n)$. The techniques developed in \cite{Eslava2020} provide bounds on the total variation distance between $Z_{\ge c\ln n}$ and a Poisson random variable with mean $n^\alpha$ with a suitable $\alpha=\alpha(c)$ for  $c\in (1, \log_2 n)$. However, there is far less control over the random variables $Z_{\ge c\ln n}$ for $c\le 1$ \cite{Addario-Eslava2017}.
The following two theorems provide precise statements to back up the informal approximation $Z_{\ge D}\sim n^{1-\ln 2}$. 

\begin{thm}
\label{teo:probability}Let $\gamma:=1-\ln(2)$, the following convergence in probability holds
\[  \dfrac{\ln(Z_{\geq D} )}{\ln(n)}  \stackrel{p}{\longrightarrow} \gamma. \]
\end{thm}

\begin{thm}
\label{teo:bounds} 
Let $\gamma:=1-\ln(2)$. For any positive integer $k$, 
\begin{align*}
\E{\big(\Ln{Z_{\geq D}}\big)^k} = \big(\g\Ln{n}\big)^k \left(1+o(1)\right).
\end{align*}
\end{thm}

Theorem~\ref{thm:attack} follows immediately from Theorem~\ref{teo:probability}. 

\begin{proof}[Proof of Theorem~\ref{thm:attack} assuming Theorem~\ref{teo:probability}]
Since the removal of vertices is done accordingly to their degree, $Z_{\geq D}$ gives us the worst-case destruction time for the targeted cutting procedure; that is, $X^{targ}_n\le Z_{\geq D}$ a.s.. 
It then follows that \eqref{eq:Kn_Op} is satisfied, for $\eps,\delta>0$, by letting $M(\eps,\delta)=1$ and $N$ large enough that $\p{Z_{\ge D}\ge n^{\gamma +\eps} }<\delta$, for $n\ge N$. \end{proof}

The proof of Theorem~\ref{teo:probability} is based on the concentration of $D$ and the first and second moment method, see Section~\ref{sec proof theo}. 
Unfortunately, the tails of $D$ do not vanish as fast as a naive 
bounding of the moments of $Z_{\ge D}$ would require 
by the current control we have on the distribution of $Z_{\ge d}$ for $d\sim \ln n$ (see \eqref{eq: concentration} and Propositions \refand{prop:moments}{prop:Poi}, respectively).
Instead, to establish \refT{teo:bounds}, we resort to a coupling between a random recursive tree $T_n$ of size $n$ and a random recursive tree $T_n^{(\eps)}$ of size $n$ conditioned on $D$ to take values in $( (1- \e) \ln (n), (1 +\e)\ln(n) )$, for any given $\eps\in (0,1)$, see Section~\ref{sec:coupling}. 

Briefly described, the coupling is the following. Let $n\in \N$, for the construction of both $T_n$ and $T_n^{(\eps)}$ it suffices to define the parent of vertex $i$, for each $1< i\le n$. For each tree, we break down this choice in two steps: First sample a random Bernoulli to decide whether $i$ attaches to the root or not. Then the parent of $i$ is either the root or a uniformly sampled vertex among the rest of the possible parents. The Bernoulli variables are coupled in such a way that $T_n$ is constructed by independent variables while the Bernoulli random variables related to the tree $T_n^{(\eps)}$ imply the conditioning on $D$ taking values in the aforementioned interval. On the other hand, if a given vertex chooses a parent distinct from the root, then this choice is \textit{exactly the same} for both the unconditioned and conditioned tree.

\subsection{Discussion on related models}\label{sec:Discussion}

Cutting processes of random trees and graphs can serve as a proxy to the resilience of a particular network to breakdowns, either from intentional attacks or fortuitous events. What is considered a breakdown and resilience may differ depending on the context. In the first cutting procedure, introduced by Meir and Moon for random recursive trees \cite{Meir1974}: it considers contamination from a certain source within an organism and one may think of the number of cuts as the necessary steps to isolate the source of the contamination. 

Janson \cite{Janson2006}, noted that the number of cuts needed to destroy a tree is equivalent to the number of records that arises from a random labeling of the edges, this approach is further used by Holmgren \cite{Holmgren2008,Holmgren2010} to study split trees and binary search trees, respectively. 

Several modifications of the uniform edge deletion process have been proposed by different authors.  Javanian and Vahini-Asl \cite{Javanian2004} modified the process with the objective of isolating the last vertex added to the network. Kuba and Panholzer \cite{Kuba2008a, Kuba2008, Kuba2013} published a series of articles focusing specifically on isolating a distinguished vertex; for example, the isolation of a leaf or isolating multiple distinguished vertices. 

In the context of Galton-Watson trees, cutting down trees has been predominantly studied from the perspective of vertex-cutting. That is, vertices are selected to be removed, rather than edges, and once a vertex has been removed, we keep the subtree containing the root; the procedure is repeated until the root is removed. Note that selecting an edge uniformly at random is equivalent to uniformly selecting a vertex other than the root. Bertoin and Miermont \cite{Bertoin2013} constructed a method to compare vertex-cutting versus edge cutting and studied the tree destruction process. Addario-Berry et al. \cite{ABerry2014} studied the case of Galton-Watson trees with critical, finite-variance offspring distribution, conditioned to have a fixed number of total progeny. Dieuleveut \cite{Dieuleveut2015} modified the vertex-cutting process so that the probability of a given vertex to be cut, at each step, is proportional to its degree. This process was generalized by Cai et al. \cite{Cai2019} by considering that every vertex can stand $k$ attacks before being removed. This generalization was later studied by Berzunza et al. \cite{Berzunza2021} for deterministic trees and by Berzunza et al. \cite{Berzunza2020} for conditioned Galton-Watson trees.

Similar cutting procedures have been studied in more general graphs. To name a few, Berche et al. \cite{Berche2009} studied public transport networks; Xu et al. \cite{Xu2019} studied a broad range of dynamical processes including birth-death processes, regulatory dynamics and epidemic processes on scale free graphs and Erdös-Renyi networks under two different targeted attacks, with both high-degree and low-degree vertices; Alenazi and Sterbenz \cite{Alenazi2015} compared different graph metrics, such as node betweenness, node centrality and node degree, to measure network resilience to random and directed attacks.

As we mentioned before, what is considered a breakdown and resilience may differ depending on the context. In the internet, for example, the failures of connectivity that happen over time can be modeled by random cuts, while malicious attacks, from a hacker or enemy trying to disconnect some given network of servers, can be mathematically described by targeted cutting towards highly connected servers. These ideas on resilience were posed by Cohen et al. \cite{Cohen2000,Cohen2001} for scale free graphs. Later, Bollob\'as and Riordan \cite{Bollobas2004} compared random versus malicious attacks on scale free graphs. Since malicious attacks may hold an strategy to better take advantage of some characteristics of the network, it is expected that the number of cuts required would be significantly fewer than in a completely random attack. 

For scale-free networks Cohen et al. \cite{Cohen2001} obtained the next result. Assume the degree of a uniformly random vertex of the graph follows a power law with decay $\alpha$ on the support $m,\dots ,n$. If a proportion $p$ of the vertices with the highest degree is deleted, then the probability $\bar{p}$ of a randomly chosen node to connect to a deleted vertex is roughly approximated by $p^{\frac{2-\alpha}{1- \alpha}}$, for $\alpha >2$. In the case $\alpha=2$ the approximation of $\bar{p}$ is given by $\ln \Big( \frac{np}{m} \Big)$. This result supports the intuition of needing a small proportion of vertices to be removed in a targeted attack to delete a network.

\subsection{Notation}

We use $|A|$ to denote the cardinality of a set $A$. For $n<m\in \N$ we write $[n]:=\{1,2,\dots, n\}$ and $[m,n]=\{m,m+1,\ldots, n\}$. For $a,b, c,x\in \R$ with $b,c>0$, we use $x\in(a\pm b)c$ as an abbreviation for $(a-b)c\leq x\leq (a+b)c$. In what follows we denote natural and  base $2$ logarithms by $\ln (\cdot)$ and $\log (\cdot)$, respectively. We often use the identity $\ln (a) =\ln (2)\log (a)$ for $a>0$.

For $r\in\R$ and $a\in\N$ define $(r)_a:=r(r-1)\cdots (r-a+1)$ and $(r)_0=1$. For real functions $f,g$ we write $f(x)=o(g(x))$ when $\lim_{x\to \infty}  f(x)/g(x) =0$ and $f(x)=O(g(x))$ when $|f(x)/g(x)|\le C$ for some $C>0$. The convergence in probability will be written as $\stackrel{p}{\longrightarrow}$. A rooted tree is a tree with a distinguished vertex, which we call the root. We always consider the edges to be directed towards the root. A directed edge $e=uv$ is directed from $u$ to $v$ and, in this case, we will say that $v$ is the parent of $u$. Given a rooted tree $T$ and one of its vertices $v$, the degree of $v$ in $T$, denoted by $d_T(v)$, is the number of edges directed towards $v$. We say that $T$ has size $n$ if it has $n$ vertices, and $[n]$ will denote the set of vertices of a tree of size $n$.

\section{Deterministic tails of the degree sequence}\label{sec proof theo}

Given a random recursive tree $T_n$, let $Z_{\geq d}$ denote the number of vertices with degree at least $d$, that is
\begin{align}\label{dfn:Z}
Z_{\geq d} \equiv Z_{\geq d}(n):=\big\vert \{v\in [n]: d_{T_n}(v)\geq d \}\big\vert. 
\end{align}
Our theorems build upon results on the convergence of the variables $Z_{\ge d}$ since, they are non-increasing on $d$ and $Z_{\ge D}$ is a tail of the degree sequence with random index; recall that $D=D(n)$ denotes the degree of the root of $T_n$. The following two propositions are simplified versions of Proposition 2.1 in \cite{Addario-Eslava2017} and Theorem $1.8$ in \cite{Eslava2020}, respectively.

\begin{prop}[Moments of $Z_{\ge d}$]
\label{prop:moments}
For any $d\in \N$, $\E{Z_{\ge d}}\le 2^{\log (n)-d}$. Moreover, there exists $\alpha>0$ such that if $d < \frac{3}{2}\ln (n)$ and $k\in \{1,2\}$ then 
\begin{align}
    \E{(Z_{\ge d})_k} &= (2^{\log (n)-d})^k(1+o(n^{-\alpha})), \label{eq:firstmoment}\\
    \E{Z_{\geq d}^2} &= \E{ Z_{\geq d}}^2 \big( 1+ o(n^{-\alpha}) \big).\label{eq:PZ}    
\end{align}
\end{prop}

\begin{prop}[Total variation distance]
\label{prop:Poi}
Let $0<\eps<\frac{1}{3}$, then for $(1+\e)\ln(n)< d< (1+3\e)\ln(n)$ there exists $\alpha'=\alpha'(\eps)>0$ such that
\[\dTV(Z_{\geq d}, \Poi{\E{Z_{\geq d}}}) \leq O(n^{-\alpha'}).\]
\end{prop}  

As we mentioned before, the error bounds in the previous propositions are not strong enough to estimate the moments of $Z_{\ge D}$. Instead we focus on the variable $\ln(Z_{\ge D})$. Furthermore, we will 
transfer the task of moment estimation, for 
Theorem~\ref{teo:bounds}, to $\E{(\ln X)^\ell}$ where instead of having $X=Z_{\ge D}$ we consider either $1+Z_{\ge m_-}$ or $1+Z_{\ge m_+}$ for suitable values  $m_\pm=c_\pm\ln n$ with $c_-<1<c_+$.

The upper bound in the following proposition follows from a straightforward application of Jensen's inequality, together with \refP{prop:moments}; while the lower bound uses the refined bounds given in Proposition~\ref{prop:Poi}.

\begin{prop}\label{prop:B}
Let $\e\in(0,\frac{1}{3})$ and $\ell\in\N$. Let $m_-:= \lfloor (1-\frac{\e}{2\ln(2)})\ln(n) \rfloor$, $m_+:= \lceil (1+\frac{\e}{2\ln(2)})\ln(n) +1 \rceil$. There exist $\alpha'=\alpha'(\eps)>0$ and $C_\ell=C_\ell(\eps)>0$ such that 
\begin{align}
\E{\big(\Ln{1+Z_{\geq m_-}} \big)^\ell} &\leq \left[\big(1-\ln(2)+\e\big)\ln(n) \right]^\ell +C_\ell,  \label{eq:upperB}\\  
\E{\big(\Ln{1+Z_{\geq m_+}} \big)^\ell} &\geq \left[\left(1-\ln(2)-\e\right)\ln(n) \right]^\ell \left(1-O(n^{-\alpha'}) \right). \label{eq:lowerB}
\end{align}
\end{prop}
\begin{proof}
First, let $f(x)=(\ln(x))^\ell$ and note that $f''(x)\leq 0$ for $x>e^{\ell-1}$. Hence, by Jensen's inequality, for any non negative random variable $X$ it holds
 
\[\E{(\ln(X))^\ell}= \E{(\ln(X))^\ell \mathbf{1}_{\{X> e^{\ell -1}\}}} + \E{(\ln(X))^\ell \mathbf{1}_{\{X\leq e^{\ell-1}\}}}  \leq  (\ln( \E{X}))^\ell + (\ell-1)^\ell.  \]

Next we will use the upper bound in \refP{prop:moments} for $\E{Z_{\geq m_-}}$.
Note that $\log (n) - m_-\le   (1-\ln(2)+ \frac{\e}{2})\log(n)$. Thus, $\E{1+Z_{\geq m_-}}\le 1+2n^{1-\ln(2)+\frac{\e}{2}}\le n^{1-\ln(2)+\e}$; where the second inequality holds for $n_0=n_0(\eps)$ large enough. Then
\begin{align*}
    \E{\left(\ln(1+Z_{\geq m_-}) \right)^\ell} &\leq \left[\ln\left(\E{1+Z_{\geq m_-}}\right) \right]^\ell + (\ell-1)^\ell
    \leq \left[(1-\ln(2)+\e)\ln(n) \right]^\ell + C_\ell; 
\end{align*}
where $C_\ell = \sup_{n\leq n_0}\{(\ln(1+2n^{1-\ln(2)+\frac{\e}{2}}))^\ell\}+(\ell-1)^\ell$. 

Next, let $\mu=\E{Z_{\geq m_+}}$ and let $(X,X')$ be random variables coupled so that $X\eqdist Z_{\geq m_+}$, $X'\eqdist \Poi{\mu}$ and $\p{X=X'}$ is maximized; that is, $\p{X\neq X'}=\dTV(X,X')$. Note that 
\begin{align*}
    \E{(\ln(1+X))^\ell}\ge \E{(\ln(1+X))^\ell \I{X> n^{\g-\e}}}\ge \left(\ln \left(n^{\g-\e}\right)\right)^\ell\p{X>n^{\g-\e}};
\end{align*}
then \eqref{eq:lowerB} boils down to lower bounding $\p{X>n^{\g-\e}}$. Since $X<n$, by the coupling assumption, we have 
\begin{align*}
    \p{X>n^{\g-\e}}&\ge \p{n^{\g-\e}<X'<n}-\p{X\neq X'}\\
    &=1-\p{X'\ge n}-\p{X'\le n^{\g-\e}}-\dTV(X, X'). 
\end{align*}
By \refP{prop:Poi}, $\dTV(X, X')=O(n^{-\alpha'})$ for $\alpha'(\e)>0$. Using the Chernoff bounds for the tails of a Poisson variable (see, e.g. Section 2.2 in \cite{Concentration}) and that $\mu=n^{\g-\frac{\e}{2}}(1+o(1))$ we have 
\begin{align*}
    \p{X'\ge n} &\le \left(\frac{en^{\g-\frac{\e}{2}}}{n}\right)^n e^{-n^{\g-\frac{\e}{2}}} \le \left(\frac{ e^{n^{\ln(2)}}}{en^{\ln(2)}}\right)^{-n}, \\
    \p{X'\le n^{\g-\e}} &\le \left(\frac{en^{\g-\frac{\e}{2}}}{ n^{\g-\e}} \right)^{n^{\g-\e}} \hspace{-1em} e^{-n^{\g-\frac{\e}{2}}} \le \left(\frac{ e^{n^{\e/2}}}{en^{\frac{\e}{2}}}\right)^{-n^{\g-\e}};
\end{align*}
both bounds are $o(n^{-\alpha'})$ so the proof is completed.
\end{proof}

\subsection{Proof of Theorem \ref{teo:probability}}
 
We will use the first and second moment method, together with \refP{prop:moments}, the concentration of $D$ and the fact that for each $n$, $Z_{\geq m}$ is non-increasing in $m$. 

For completeness, we show that $D$ is concentrated around $\ln n$. Indeed, $D$ is a sum of independent Bernoulli random variables $(B_i)_{1<i\le n}$, each with mean $1/(i-1)$ and so $\E{D}=H_{n-1}> \ln n$ where $H_n$ denotes the $n$-th harmonic number. From the fact that $H_n-\ln n$ is a decreasing sequence we infer that: for any $0<\eps<3/2$ and $n$ sufficiently large,  $|D-H_{n-1}|\le \frac{\eps}{2} H_{n-1}$ implies $|D-\ln n|\le \eps \ln n$. Using the contrapositive of such statement and Bernstein's inequality (see, e.g. Theorem 2.8 in \cite{Janson2000}) we obtain, for $n$ large enough,
\begin{equation}\label{eq: concentration}
\p{ D \notin (1\pm\e)\ln(n)}\leq \p{\left\vert D- H_{n-1} \right\vert >  \dfrac{\e}{2} H_{n-1} }\leq 2\exp \left\{-\dfrac{\e^2}{12}H_{n-1} \right\}\leq 2 n^{-\e^2/12}.
\end{equation}

Recall $\gamma=1-\ln(2)$. It suffices to prove that for every $\e>0$,
 \[ \lim_{n\to \infty} \p{ Z_{\geq D}\notin (n^{\gamma-\e},n^{\gamma+\e}) } = 0.\]
We infer from \eqref{eq: concentration} that $\p{Z_{\geq D}\notin (n^{\gamma-\e},n^{\gamma+\e}), D\notin (1\pm \e)\ln(n)}$ vanishes as $n$ grows. 
 
Let $m_- := m_-(\e ) = \lfloor (1-\e)\ln(n) \rfloor$ and $m_+:= m_+(\e) =\lceil (1+\e)\ln (n) \rceil$.
Using the monotonicity of $Z_{\ge m}$ on $m$, we have
\begin{equation}
\label{eq: prop}
\p{Z_{\geq D}\notin (n^{\gamma-\e},n^{\gamma+\e}),\, D\in (1\pm \e)\ln(n) } \leq \p{ Z_{\geq m_-}\geq n^{\gamma+\e} }+\p{ Z_{\geq m_+}\leq n^{\gamma-\e}};
\end{equation}
so it remains to show that both terms in the right side of (\ref{eq: prop}) vanish. First, using that $\ln (n)=\ln(2)\log n$ and so  $\log(n)-(1\pm\eps)\ln n=(1-\ln(2)\mp\eps\ln(2))\log(n) $, we infer from \refP{prop:moments} that 
\begin{align}\label{eq:momentconseq}
   \frac{1}{2}n^{\gamma-\eps\ln(2)}(1-o(n^{-\alpha})) \le \E{Z_{\geq m_+}} \le \E{Z_{\geq m_-}} \leq 2n^{ \gamma+\eps\ln(2)}.
\end{align}
Markov's inequality then gives
$\p{Z_{\geq m_-}\geq n^{\gamma+\e}} \leq 2n^{-\e\g}\to 0$. 
Next, let $\theta$ be defined so that $n^{\gamma-\eps}=\theta \E{Z_{\geq m_+}}$; in particular, $\theta\le 2n^{-\eps\gamma}$. Paley-Zygmund inequality gives
\begin{align*}
\p{Z_{\geq m_+}  > n^{\gamma-\e}} \ge \p{Z_{\geq m_+} > \theta\E{Z_{\geq m_+}}} \geq (1-\theta)^2\dfrac{\E{Z_{\geq m_+}}^2}{\E{Z_{\geq m_+}^2}}; 
\end{align*}
which tends to 1 as $n\to \infty$ by the upper bound for $\theta$ and \eqref{eq:PZ}. This implies $\p{Z_{\geq m_+}\le n^{\gamma-\eps}}$ vanishes, as desired.

\section{Control on $D$ through a coupling}\label{sec:coupling}

Let $n\in \N$. Write $\cI_n$ for the set of increasing trees of size $n$; namely, labelled rooted trees with label set $[n]$ such that vertex labels are increasing along any path starting from the root. It is straightforward to verify that the law of $T_n$ is precisely the uniform distribution on $\cI_n$. 

Consider the following construction of an increasing tree of size $n$. Let $(b_i)_{1<i\le n}$ and $(y_i)_{1<i\le n}$ be integer-valued sequences such that $b_2=y_2=1$,  $b_i\in \{0,1\}$ and $2\le y_i\le  i-1$ for $3\le i\le n$. 
Let vertex labelled 1 be the root and, for each $1<i \le n$, let vertex $i$ be connected to vertex $1$ if $b_i=1$ and, otherwise, let vertex $i$ be connected to vertex $y_i$.  

The following coupling is exploited in the proof of Theorem~\ref{teo:bounds}.
Define random vectors $(B_i)_{1<i\le n}, (B_i^{(\e)})_{1<i \le n}, (Y_i)_{1<i\le n}$ as follows. Let $(B_i)_{1<i \le n}$ be independent $\Ber{\frac{1}{i-1}}$ random variables, let $(B_i^{(\e)})_{1<i \le n}$ have the law of $(B_i)_{1<i \le n}$ conditioned on $\sum_{i=2}^{n} B_i \in (1\pm\e)\ln(n)$ and let $(Y_i)_{1<i\le n}$ be independent random variables such that $Y_2=1$ a.s. and $Y_i$ is uniform over $\{2, \ldots, i-1\}$ for $3\le i\le n$. We assume that the vector $(Y_i)_{1<i\le n}$ is independent from the rest, while the coupling of $(B_i)_{1<i\le n}$ and $(B_i^{(\e)})_{1<i \le n}$ is arbitrary. 

The tree obtained from $(B_i)_{1<i\le n}, (Y_i)_{1<i\le n}$ and the construction above has the distribution of a RRT. To see this, write $v_i$ for the parent of vertex $i$; then $1\le v_i<i$ for each $1<i\le n$. First, note that each $v_i$ is independent from the rest since $(B_i)_{1<i\le n}$ and $(Y_i)_{1<i\le n}$ are independent. Next we show that $v_i$ is chosen uniformly at random from $\{1,2,\dots,i-1\}$. First, we have $v_2=1$ almost surely. For $2\le \ell< i\le n$, by the independence of $B_i$ and $Y_i$, we have 
\begin{align*}
    \p{v_i=1}&= \p{B_i=1}=\frac{1}{i-1}=\p{B_i=0, Y_i=\ell}=\p{v_i=\ell};
\end{align*}
therefore, the tree obtained has the law of a RRT and so we denote it by $T_n$. Analogously, write $T_n^{(\e)}$ for the tree obtained from $(B_i^{(\e)})_{1<i \le n}$ and $(Y_i)_{1<i\le n}$, and let $D^{(\e)}$ be the degree of its root.

By definition of $D$ and the construction above we have $D=\sum_{i=2}^n B_i$. Thus, conditioning on $\sum_{i=2}^n B_i$ means, under this construction, to condition $T_n$ on the root degree $D$. 
In particular, the distribution of $(B_i^{(\e)})_{1< i\le n}$ is defined so that $T_n^{(\e)}$ has the distribution of a RRT of size $n$ conditioned on $D^{(\e)}\in (1\pm\e)\ln(n)$.

Since $D$ is concentrated around $\ln (n)$, the conditioning on $T_n^{(\e)}$ is over an event of probability close to one and so the degree sequences of $T_n$ and $T_n^{(\e)}$ do not differ by much. Hence, the proof strategy of \refT{teo:bounds} is to estimate the moments $\E{(\ln(Z_{\ge D}))^k}$ using the monotonicity of $Z_{\geq d}$, by conditioning on $D\in (1\pm \e) \ln (n)$ while retaining $Z_{\geq d}$ instead of  $Z_{\geq d}^{(\e)}$ ; see \eqref{eq:AppCoupling}--\eqref{eq:logsum2}.

The following two propositions makes this idea rigorous. For $d\ge 0$, let 
\begin{align}
W_d&=\frac{1+Z_{\geq d}^{(\e)}}{1+Z_{\geq d}}    
\end{align}
where $Z_{\geq d}$ is defined as in \eqref{dfn:Z}
 and, similarly, let 
$  Z_{\geq d}^{(\e)}:=|\{v\in T_n^{(\e)} :\, d_{T_n^{(\e)}}(v)\geq d\}|$. 

The key in the proof of \refP{prop:Coupling} lies on \eqref{eq:Coupling2}, which yields an upper bound on the number of vertices that have differing degrees in $T_n$ and $T_n^{(\e)}$ under the coupling. In turn, this allows us to infer bounds on the ratio $W_d$ that hold with high probability and are uniform on $d$. 

\begin{prop} 
\label{prop:Coupling}
Let $\e,\delta\in(0,1)$ and $0\le d \leq (1+\e)\Ln{n}$. There is $C>0$,  $\beta=\beta(\eps)$ and $n_0=n_0(\delta)$ such that for $n\geq n_0$, under the coupling described above, we have 
\[\p{W_d\in (1\pm \delta)}\geq 1-Cn^{-\beta}.\] 
\end{prop}

\begin{proof}
Let $m_+=\lceil(1+\eps)\ln n\rceil$ so that $Z_{\geq d} \geq Z_{\geq m_+}$. By Chebyshev's inequality and \eqref{eq:PZ}, for $c\in (0,1)$, 
\begin{align}\label{eq:Chebyshev}
\p{ Z_{\geq m_+} \le c\E{Z_{\ge m_+}}} \leq \p{ \left\vert Z_{\geq m_+}-\E{Z_{\geq m_+}}\right\vert\geq (1-c)\E{Z_{\ge m_+}} }= o(n^{-\alpha}).
\end{align}

Rewrite $W_d=1+\frac{Z_{\ge d}^{(\e)}-Z_{\ge d}}{1+Z_{\ge d}}$ to see  
$W_d\in (1\pm \delta)$ is equivalent to 
$|Z_{\ge d}^{(\e)}-Z_{\ge d}|\in [0,\delta (1+Z_{\ge d}))$.
Hence, it suffices to show that there is $n_0(\delta)\in \N$, such that for $n\ge n_0$, 
\begin{align}\label{eq: Coupling}
    \{Z_{\ge m_+}\ge c \E{Z_{\ge m_+}}, D\in (1\pm \e)\ln(n)\}\subset\{|Z_{\ge d}^{(\e)}-Z_{\ge d}|\leq\delta(1+ Z_{\ge d})\};
\end{align}
which by a contrapositive argument, together with \eqref{eq: concentration} and \eqref{eq:Chebyshev}, yields for $\beta=\min{\{\alpha,\e^2/12\}}>0$,
\begin{align*}
    \p{|Z_{\ge d}^{(\e)}-Z_{\ge d}|>\delta(1+ Z_{\ge d})} 
    &\le \p{D\notin (1\pm \e)\ln(n)} +\p{Z_{\ge m_+}\le c \E{Z_{\ge m_+}}}
    = O(n^{-\beta}). 
\end{align*}
We extend the notation introduced for the coupling; let $\cS:=\{i\in [2,n]: v_i=v_i^{(\e)}\}$
be the set of vertices that have the same parent in $T_n$ and $T_n^{(\e)}$ and for $i\in [n]$ denote the set of children of $i$ in $T_n$ and $T_n^{(\e)}$, respectively, by
\begin{align*}
    \cC(i):=\{j\in [2,n]: v_j=i\}\qquad \text{and}\qquad
    \cC^{(\e)}(i):=\{j\in [2,n]: v_j^{(\e)}=i\}.
\end{align*}
By the coupling construction, $\cS \cup (\cC(1)\bigtriangleup \cC^{(\e)}(1))$ is a partition of $[2,n]$; that is, whenever the parent of a vertex $i\in [2,n]$ differs in $T_n$ and $T_n^{(\e)}$ we infer $B_i\neq B_i^{(\e)}$ and so either $i\in \cC(1)\setminus \cC^{(\e)}(1)$ or $i\in \cC(1)^{(\e)}\setminus \cC(1)$. The consequence of this observation is two-fold: First, for any $i\in [2,n]$, a necessary condition for $d_{T_n}(i)\neq d_{T_{n}^{(\e)}}(i)$ is that $\cC(i)\neq \cC^{(\e)}(i)$. Second, the function $j\mapsto Y_j$ that maps  $\cC(1) \bigtriangleup \cC^{(\e)}(1)$ to $\{i\in [2,n] : \cC(i)\neq \cC^{(\e)}(i) \}$ is surjective.
Indeed, if $i\neq 1$ and $\cC(i)\neq \cC^{(\e)}(i)$ then there exists $j\in \cC(1) \bigtriangleup \cC^{(\e)}(1)$ such that $Y_j=i$. 
Together they imply the following chain of inequalities,  
\begin{equation}\label{eq:Coupling2}
|\{i\in [2,n] : d_{T_n}(i)\neq d_{T_{n}^{(\e)}}(i) \}|\le |\{i\in [2,n] : \cC(i)\neq \cC^{(\e)}(i) \}|\le \vert \cC(1) \bigtriangleup \cC(1)^{(\e)} \vert;
\end{equation}
the first inequality by containment of the corresponding sets and the second one by the surjective function described above. On the other hand, $\vert Z_{\geq d}-Z_{\geq d}^{(\e)}\vert$ equals
 \begin{align*}
\left\vert \sum_{i=1}^n 1_{ \{ d_{T_n} (i) \geq d  \} } (i) - \sum_{i=1}^n 1_{ \{ d_{T_n^{(\e)} \! } (i)  \geq d \} } (i) \right\vert 
   & \leq   \sum_{i=1}^n  \Big\vert  1_{ \{ d_{T_n} (i) \geq d  \} } (i) - 1_{ \{ d_{T_n^{(\e)} \! } (i) \geq d  \} } (i) \Big\vert  \\
   & \leq  1 + |\{ i \in \{ 2,...,n \} : d_{T_n}(i)\neq d_{T_n^{(\e)}}(i) \}|.  
 \end{align*}
Therefore, 
\begin{align}\label{eq:Coupling3}
 \vert Z_{\geq d}-Z_{\geq d}^{(\e)} \vert  \leq  1 + \vert \cC(1) \bigtriangleup \cC^{(\e)}(1) \vert \le 1+2\max\{|\cC(1)|, |\cC^{(\e)}(1)|\}.  
\end{align}

We are now ready to prove \eqref{eq: Coupling}. Fix, e.g., $c=1/2$ and let $n_0=n_0(\delta)$ be large enough that $\frac{1+4\ln(n)}{\delta}<c\E{Z_{\geq m_+}}$; this is possible since $\E{Z_{\geq m_+}}$ grows polynomially in $n$, by \refP{prop:moments}. In particular, recalling $Z_{\ge d}\ge Z_{\ge m_+}$, for $n\ge n_0$ and any $\eps\in (0,1)$ we have 
\begin{align}\label{sets:n0}
\{Z_{\ge m_+}\geq c\E{Z_{\ge m_+}}\}\subset \left\{Z_{\geq d} \geq \frac{1+2(1+\e)\ln(n)}{\delta}\right\}.     
\end{align}
Moreover, by the construction of $T_n^{(\e)}$, $|\cC(1)^{(\e)}|=D^{(\e)}\le (1+\e)\ln n$, so that \eqref{eq:Coupling3} implies
\begin{align}\label{sets:coupling1}
    \left\{Z_{\geq d} \geq \frac{1+2(1+\e)\ln(n)}{\delta}, D\in (1\pm \e)\ln(n)\right\}\subset\{|Z_{\ge d}^{(\e)}-Z_{\ge d}|\leq\delta(1+ Z_{\ge d})\};
\end{align}
note that \eqref{sets:coupling1} holds for all $n\in \N$. Together with \eqref{sets:n0}, implies \eqref{eq: Coupling} for $n\ge n_0$, as desired. 
\end{proof}

\begin{prop}\label{prop:expectation coupling}
Let $\e\in (0,1)$, $\ell \in \N$. For $0\le d \leq (1+\e)\Ln{n}$, there is $\beta=\beta(\eps)>0$  
such that, under the coupling described above, we have 
\[ \left\vert \E{(\Ln{W_d})^\ell } \right\vert \leq  (\Ln{n})^\ell O(n^{-\beta})+1. \]
\end{prop}

\begin{proof}
We first simplify to consider 
\begin{align*}
    \left\vert \E{(\Ln{W_d})^\ell } \right\vert\le \E{\left| (\ln (W_d))^\ell\right|}\le \E{|\ln(W_d)|^\ell}.
\end{align*}
Now, $\frac{1}{n}\leq W_d\leq n$ for every $0 \le d\le (1\pm\eps)\ln (n)$, since $W_0\equiv 1$ while  $\max\{Z_{\ge d}, Z_{\ge d}^{(\e)}\}<n$ for $d\ge 1$, as in any tree there is at least one vertex of degree zero. Then, for any $\delta \in(0,1)$, we have 
\begin{align*}
    \E{|\ln(W_d)|^\ell\I{W_d<1-\delta}}\le (\ln(n))^\ell \p{W_d<1-\delta},\\        \E{|\ln(W_d)|^\ell\I{W_d>1+\delta}}\le (\ln(n))^\ell \p{W_d>1+\delta}.
\end{align*}
\refP{prop:Coupling} implies these two terms are $(\ln(n))^\ell O\left( n^{-\beta}\right)$, where the implicit constant depends on the choice of $\delta$. With foresight fix $\delta$ to satisfy  $\left(\frac{\delta}{1-\delta}\right)^\ell +\delta^\ell=1$. Using that $x\ge 0$ satisfies $1-\frac{1}{x} \leq \ln(x)\le x-1$, 
\begin{align*}
    \E{|\ln (W_d)|^\ell \I{W_d\in (1\pm \delta)}}
    &\le 
    \E{\Big| 1- \frac{1}{W_d} \Big|^\ell \I{W_d\in (1- \delta, 1)}} +
    \E{|W_d-1|^\ell \I{W_d\in [1, 1+ \delta)}} \\
     & \le \Big( \frac{\delta}{1-\delta} \Big)^\ell+ \delta^\ell
    =1;
\end{align*}
and so the result follows. 
\end{proof}

\subsection{Proof of Theorem \ref{teo:bounds} }

Fix $k\in \N$ and recall $\gamma:=1-\ln(2)$. Suppose that for any $\eps\in (0,\frac{1}{3})$ there exists $c=c(\eps)>0$ and $C=C(k,\e)>0$ such that
\begin{align}\label{eq:bounds2}
   \big((\g-\e)\Ln{n}\big)^k \big(1+O(n^{-c})\big) -C \leq \E{\big(\Ln{Z_{\geq D}}\big)^k}\leq \big((\g+\e)\Ln{n}\big)^k +C. 
\end{align}
It is straightforward to verify that \eqref{eq:bounds2}  establishes \refT{teo:bounds}. So it remains to prove \eqref{eq:bounds2}. 

Let $\eps'=\eps/(2\ln(2))$ and write 
\begin{align}\label{dfn:mpm2}
    m_-= \lfloor (1-\e')\ln(n) \rfloor  \quad \text{and} \quad  m_+=\lceil (1+\e')\ln (n)+1 \rceil.
\end{align}

We focus on the term 
$\E{(\Ln{Z_{\geq D}})^k\mathbf{1}_{\{D\in (1\pm\e')\ln (n) \}}}$ as \eqref{eq: concentration} and $1\le Z_{\ge D}\le n$ imply 
\begin{equation}\label{eq:negligmoments}
    0\le \E{(\Ln{Z_{\geq D}})^k\mathbf{1}_{\{D\notin (1\pm\e')\ln (n) \}}} \le (\Ln{n})^k \p{D\notin (1\pm\e')\Ln{n}}=(\Ln{n})^k O(n^{-\eps'/12}).
\end{equation}

Using the monotonicity of $Z_{\ge d}$, we have
\begin{align*}
    \E{(\Ln{Z_{\geq D}})^k\mathbf{1}_{\{D\in (1\pm\e')\ln (n) \}}} \le \E{(\Ln{Z_{\geq m_-}})^k\mathbf{1}_{\{D\in (1\pm\e')\ln (n) \}}} \le \E{(\Ln{Z_{\geq m_-}})^k},
\end{align*}
which by \eqref{eq:upperB} yields,
\begin{equation}\label{eq:upperkey}
     \E{(\Ln{Z_{\geq D}})^k\mathbf{1}_{\{D\in (1\pm\e')\ln (n) \}}} \le ((\gamma +\eps)\ln (n))^k +C_k.
\end{equation}

For the lower bound we consider the conditional variable $Z^{(\e')}_{\geq d}$ defined in the previous section with $d=m_+$. Observe that, if $D\in (1\pm \e')\ln (n)$ then  $Z_{\ge D}\ge 1 + Z_{m_+}$, thus we obtain
\begin{align}
    \E{(\Ln{Z_{\geq D}})^k\mathbf{1}_{\{D\in (1\pm\e')\ln (n) \}}} &\ge \E{(\Ln{1+Z_{\geq m_+}})^k\mathbf{1}_{\{D\in (1\pm\e')\ln (n) \}}} \nonumber\\
    &\ge \E{\left(\Ln{1+Z^{(\e')}_{\geq m_+}}\right)^k}\p{D\in (1\pm\e')\ln (n)}.\label{eq:AppCoupling}
\end{align}

The definition of $W_d$ gives
\begin{align}\label{eq:logsum}
\Ln{1+Z_{\geq m_+}^{(\e')}}= \Ln{\left(1+Z_{\geq m_+}^{(\e')}\right)\dfrac{1+Z_{\geq m_+}}{1+Z_{\geq m_+}}}= \Ln{1+Z_{\geq m_+}}+\Ln{W_{m_+}};   
\end{align}
similarly, for $k\ge 2$, the binomial expansion implies
\begin{align}\label{eq:logsum2}
  \left( \ln \left( 1+Z_{\geq m_+}^{(\varepsilon')} \right) \right)^k = 
 \left( \ln(1+Z_{\geq m_+}) \right)^k +  \sum_{\ell=1}^{k} \binom{k}{\ell}
   \ln(1+Z_{\geq m_+})^{k-\ell}\ln(W_{m_+})^\ell.
\end{align}

We use \eqref{eq:lowerB} for a lower bound on the expectation of the main term in these last two decompositions.  
For the error terms involving $W_{m_+}$ we use \refP{prop:expectation coupling}. If $k=1$, we directly get 
\begin{align*}
\E{\Ln{1+Z_{m_+}^{(\e')}}}&= \E{\Ln{1+Z_{\geq m_+}}}+\E{\Ln{W_{m_+}}}\\
 &\ge (\g-\e)\ln(n)  \left(1+O(n^{-\alpha'}) \right)+ \Ln{n} O(n^{-\beta})-1. 
\end{align*}

If $k\ge 2$, we control each of the terms in the sum of \eqref{eq:logsum2}. For $1\le \ell\le k$, the Cauchy-Schwarz inequality gives 
\begin{align}\label{eq:lowerk2a}
\bigg\vert \E{ \ln(1+Z_{\geq m_+})^{k-\ell}\ln(W_{m_+})^\ell }\bigg\vert &\leq \E{ \ln(1+Z_{\geq m_+})^{2(k-\ell)} }^{1/2}\E{ \ln(W_{m_+})^{2\ell} }^{1/2}.
\end{align}
The deterministic bound $Z_{\ge m_+}<n$ implies $\E{\ln(1+Z_{\ge m_-})^{2(k-\ell)}} ^{1/2} \le (\ln(n))^{k-\ell}$. On the other hand, \refP{prop:expectation coupling} yields
\begin{align}\label{eq:lowerk2b}
\E{ \ln(W_{m_+})^{2\ell} }^{1/2}
&\leq \left( (\Ln{n})^{2\ell} O(n^{-\beta})+1\right)^{1/2}\le \left(\ln(n)\right)^{\ell} O(n^{-\beta})+1.
\end{align}
Thus, after taking expectations in \eqref{eq:logsum2}, we get 
\begin{align}\label{eq:lowerkey}
    \E{\left(\Ln{1+Z^{(\e')}_{\geq m_+}}\right)^k}\ge ((\gamma-\eps)\ln (n))^k(1-o(n^{-c}))-C; 
\end{align}
where $c= \min\{\alpha',\beta\}$ and $C=\max\{C_k,2^k\}$. Then \eqref{eq:negligmoments}, \eqref{eq:upperkey}, \eqref{eq:lowerkey} together with \eqref{eq: concentration}, imply \eqref{eq:bounds2} completing the proof. 

\section{Conclusion and open problems}

As far as we know, our results provide the first quantitative estimate for the deletion time $X^{targ}_n$ for random recursive trees. Our main result, Theorem \ref{thm:attack}, confirms the intuition that the targeted procedure requires substantially fewer cuts than the random edge deletion procedure. It remains an open question whether $\ln(X^{targ}_n)$ also grows asymptotically as $\ln(n)$. Contrary to the case of uniform edge-cutting, in the targeted vertex-cutting process it is challenging to describe, at each step, the distribution of either the cut tree or the remaining tree. Even keeping track of the number of vertices in the first cut tree remains an open question.

\bibliographystyle{plain}

\end{document}